\documentclass[11pt,reqno]{amsart}

\usepackage[colorlinks=true, linkcolor=blue, anchorcolor=blue, citecolor=blue, filecolor=blue, menucolor= blue, urlcolor=blue]{hyperref}

\usepackage{amsmath,amssymb,amsthm,amscd}
\usepackage{verbatim}
\usepackage{comment}
\usepackage{multirow}
\usepackage{mathtools}
\usepackage{enumitem}

\makeatletter
\renewcommand*\env@matrix[1][*\c@MaxMatrixCols c]{%
  \hskip -\arraycolsep
  \let\@ifnextchar\new@ifnextchar
  \array{#1}}
\makeatother

\long\def\eatit#1{}
\newtheorem{thm}{Theorem}[section]
\newtheorem{prop}[thm]{Proposition}
\newtheorem{lem}[thm]{Lemma}

\theoremstyle{definition}

\newtheorem{ques}[thm]{Question}
\newtheorem{rem}[thm]{Remark}

\newcommand{\p}{p}
\newcommand{\pr}[1]{{{\mathbb P}^{#1}}}

\newcommand{\FF}{\mathbb{F}}
\newcommand{\PP}{\mathbb{P}}
\newcommand{\ZZ}{\mathbb{Z}}

\newcommand{\field}{K}
\newcommand{\lra}{\longrightarrow}

\newcommand{\rees}{\mathcal{R}}

\newcommand{\depth}{\rm \mathop{depth}}

\newcommand{\Hom}{\rm Hom}
\newcommand{\Ext}{\rm Ext}

\newcommand{\Sym}{\rm Sym}
\renewcommand{\Im}{\rm Im}
\renewcommand{\char}{\rm char}
\newcommand{\Span}{\rm Span}

\newcommand{\X}{\mathbf{ X}}
\renewcommand{\P}{\mathbb{P}}

\title [A criterion for the failure of containments]
{A homological criterion for the containment between symbolic and ordinary
powers of \\ some ideals of points in $\pr{2}$}
\author{Alexandra Seceleanu}

\begin{document}

\begin{abstract}We establish a criterion for the (failure of) the containment $I^{(m)}\subseteq I^r$ for 3-generated ideals $I$ defining reduced sets of points in $\pr{2}$. Our criterion arises from studying the minimal free resolutions of the powers of $I$, specifically the minimal free resolutions for  $I^m$ and $I^r$. We apply this criterion to two point configurations that have recently arisen as counterexamples to a question of B. Harbourne and C. Huneke: the family of Fermat configurations  and the Klein configuration.
\end{abstract}
\maketitle

\section{Comparing symbolic and ordinary powers of ideals}

Let $\field$ be a field. Every point $p$ in the complex projective $N$-space $\P^N$ over $\field$ gives rise to the ideal $I( p )\subset K[x_0,\ldots,x_N]$ generated by the homogeneous polynomials vanishing at $\p$.  A finite set $\X\subset {\bf P}^N$ of points in projective space  corresponds to  the ideal $I(\X)$ generated by all homogeneous polynomials  (forms) vanishing on $\X$. More precisely, if $\X=\{p_1,\ldots,p_n\}$, then $I(\X)=\cap_{i=1}^nI(p_i)$. We call an ideal $I(\X)$ of this form an ideal defining a reduced set of points in $\P^N$. Note that we do not allow irrelevant components, in other words our  ideal defining a reduced set of points are always understood to be  arithmetically Cohen-Macaulay. 

 If  $I(\X)=\cap_{i=1}^nI(p_i)$ is an  ideal defining a reduced set of points in $\P^N$, the $m^{th}$ {symbolic power} $I(\X)^{(m)}$ is given by 
$$I(\X)^{(m)}=\bigcap_{i=1}^nI(p_i)^m.$$ 

Symbolic powers arise as naturally in geometry as ideals specified by higher order vanishing conditions. The ideal $I(\X)^{(m)}$ defined above consists manifestly of  all forms whose order of vanishing is at least $m$ at each point of $\X$, if we define the order of vanishing of a form $F$ at a point $p_i$ to be the exponent of the largest ordinary power of $I(p_i)$ containing $F$. In fact, symbolic powers can be more generally defined for arbitrary  ideals $I$ of a ring $R$ by setting
$$I^{(m)}=\bigcap_{P\in\rm{Min}(I)}\left(I^mR_P\cap R\right),$$ although the special case of ideals of points presented first will suffice for our purposes.  With this more general definition, it is still true, for prime ideals defining varieties in projective space over an algebraically closed field, that the $m$-th symbolic power consists of  all forms whose order of vanishing is at least $m$ at each point of the variety, by a classical theorem of Zariski and Nagata \cite[Theorem 3.14]{refEis}. 

While the definitions above indicate a close relationship between symbolic powers and ordinary powers of ideals, it is natural to ask: how do the various symbolic powers and ordinary powers of the same ideal compare with respect to containment? This is a difficult and broad question that can also be asked for arbitrary homogeneous ideals. A celebrated answer to this question  was given in \cite{ELS,HoHu} using multiplier ideal techniques and positive characteristic methods respectively:
\begin{thm}[Ein-Lazarsfeld-Smith \cite{ELS}, Hochster-Huneke \cite{HoHu}]
The containment  ${ I^{(rN)}\subseteq I^r}$ holds for all homogeneous ideals $I\subseteq\field[x_0,\ldots,x_N]$ and for all $r\geq 1$.
\end{thm}

What currently remains open is whether the containment described before is best possible. In an attempt to strengthen this containment, one may try to make the symbolic side larger by decreasing the symbolic exponent. As a specific  case of this approach for $N=2$ and $r=2$, C. Huneke asked: 

\begin{ques}
Does  ${I^{(3)}\subseteq I^2}$  hold  for any radical ideal $I$ defining a finite set of points of $\pr2$?
\end{ques}

 An extended version of the above question  was proposed  by B. Harbourne, who asked whether $I^{(Nr-N+1)}\subseteq I^r$ would hold for all $r\geq 1$ and all $N\geq 1$ for ideals $I\subseteq \field[x_0, \ldots x_N]$ defining finite sets of reduced points \cite[Conjecture 4.1.1]{HaHu}. While the same questions can be asked more generally for arbitrary homogeneous ideals, the case of points was proposed as a fertile testing ground that benefits from being supported by geometric intuition. To give added validity to the above questions, it is worth noting that both have affirmative answer when $I$ is the ideal defining a general set of points in $\pr 2$ (\cite[Remark 4.3]{refBH}).

However, through recent developments, it is now known that the containment ${I^{(3)}\subseteq I^2}$ does not always hold. We now give an account of the successive discovery of counterexamples that have arisen in relation to this containment. Understanding these counterexamples constitutes the main motivation of this paper. The first proof that the containment ${I^{(3)}\subseteq I^2}$ can fail, appeared in the paper \cite{refDST}. The  counterexample given therein is a configuration of 12 points in the complex projective plane,  defined by the ideal $$I=(x(y^3-z^3), y(z^3-x^3), z(x^3-y^3)).$$
These points arise as the singular locus of an arrangement of 9 lines. The entire incidence structure of the 12 points and 9 lines is projectively dual to the classical Hesse configuration given by the 9 flexes of a general plane cubic together with the 12 lines passing through pairs of these flexes. A characteristic 3 analogue of the counterexample of \cite{refDST} was later given in \cite{refBCH}. More precisely, it is proven therein that removing any single point from among the 13 points of the finite projective plane $\PP_{\mathbb{F}_3}^2$ together with all the lines passing through the removed point yields on the remaining 12 points the same incidence structure exhibited by the dual Hesse configuration of \cite{refDST}. It is worth noting, however, that although they share  the same combinatorial data and the property that $I^{(3)}\not\subseteq I^2$, the dual Hesse configuration of \cite{refDST} and the characteristic 3 counterexample of \cite{refBCH} behave very differently when viewed from the perspective of  algebraic-geometric  invariants associated to them (see \cite{refDHNSST} for a detailed account of the differences and computations of the  resurgence for these counterexamples).

  Generalizing the work of \cite{refDST}, an  infinite family  of counterexamples was given in \cite{refHS}: over any field $K$ that contains $j$ distinct $n$-th roots of 1, the ideal $$I=(x(y^n-z^n), y(z^n-x^n), z(x^n-y^n)$$ defines a set of $n^2+3$ points that arises as the singular locus of an arrangement of $3n+3$ lines. Following \cite{refDHNSST}, which uses terminology introduced by Urzu\'a \cite{refUz}, we refer to these  configurations  of $n^2+3$ points as Fermat configurations. It is proven in \cite[Proposition 2.1]{refHS} that, if $I$ is the defining ideal of a Fermat configuration of points, then ${I^{(3)}\not\subseteq I^2}$. 

Recently, the authors of \cite{Bea} have found two new configurations of planar points that fail to satisfy the containment ${I^{(3)}\subseteq I^2}$. These are both configurations of points studied classically, one by Klein \cite{refKl} in conjunction with a certain quartic curve and the other by Wiman \cite{refWi} in relation to the group of collineations of the projective plane. In the following, we refer to each of these configurations using the name of their respective discoverer. The Klein configuration consists of 49 points which form the singular locus of an arrangement of  21 lines.  The Wiman configuration  consists of 201 points that arise as the singular locus of an arrangement of 45 lines. Explicit coordinates of the points in each of these configurations (over $\mathbb{C}$) are given in \cite{Bea}, where it is also noted that the failure of the containment ${I^{(3)}\not\subseteq I^2}$ was verified  computationally for both configurations by the authors.

Motivated by these counterexamples,  we propose a homological approach meant to verify the failures of containment mentioned above from  a theoretical perspective. Our end goal is to produce a criterion for the failure of the containment $I^{(3)}\subseteq I^2$ that would apply to the aforementioned counterexamples. We hope that such an approach will give some insight into what makes these configurations special from the point of view of the relation between their symbolic and ordinary powers. 
It is clear, by definition, that the Fermat configurations of points mentioned above are defined by 3-generated ideals. We prove  in section \ref{sKlein} that the Klein configuration is also cut out by a 3-generated ideal with three minimal generators of degree 8. Motivated by this, we restrict our attention to 3-generated ideals of height 2, with minimal generators of the same degree: $$I=(f,g,h)\subset \field[x,y,z].$$  

Our paper is structured as follows: we determine the minimal free resolutions of the powers of $I$, specifically the minimal free resolutions for $I^2$ and $I^3$ in section \ref{sect2}. We establish a homological criterion for   containments of the from $I^{(m)}\subseteq I^r$  in section Proposition \ref{inj}. When applied for $m=3$ and $r=2$, this criterion can be expressed concretely in terms of a single map in the aforementioned  minimal free resolution of $I^3$. Furthermore, the map in question is represented by a certain matrix easily described in terms of the minimal syzygies on the generators of $I$ by Theorem \ref{main}.  We apply our criterion to the Fermat point configurations in section \ref{sFermat}, giving a new theoretical proof that  $I^{(3)}\not\subseteq I^2$. We further apply our criterion to the classical Klein configuration over the complex numbers and some analogous configurations defined over fields of prime characteristic in section \ref{sKlein}, giving the first theoretical proof that  $I^{(3)}\not\subseteq I^2$ for the defining ideal $I$ of the Klein configuration.

 \section{Resolutions of powers of uniformly 3-generated ideals of points}\label{sect2}
 
In the rest of the note, we use the term {\em strict almost complete intersection} to mean an ideal of height $h$ that has a minimal set of generators of cardinality $h+1$. An  ideal defining a reduced sets of points in $\pr{2}$, is a strict almost complete intersections if it is  3-generated. 
We start by describing  the structure of the minimal free resolutions of the second and third ordinary powers of ideals of planar points that are  strict almost complete intersections with minimal generators of the same degree. 

Let $I=(f,g,h)\subset R=K[x,y,z]$ be a homogeneous ideal with minimal generators of same degree $d$.  A useful tool for obtaining free resolution for  powers of an ideal $I$, introduced in \cite{refCHT}, \cite{refKo}, \cite{refTr}, is to consider the Rees algebra of $I$, which is defined as $\rees(I)=\oplus_{i\geq0} I^it^i$.   From the universal property of symmetric algebras, one deduces the existence of a canonical  surjective graded ring homomorphism $\Sym(I)\stackrel{}{\longrightarrow} \rees(I)$, where $\Sym(I)$ is the symmetric algebra of $I$. It is often convenient to have explicit presentation for the symmetric and Rees algebras. 
Let $S=R[T_1,T_2,T_3]$ denote a bigraded polynomial ring, where the variables of $R$ have degree $(1,0)$ and $\deg(T_i)=(d,1)$.  We write $S_{(a,b)}$ for the bidegree $(a,b)$ component of $S$ and we set $S_{(*,b)}=\bigoplus_{a\geq0}S_{(a,b)}$. Note that $S_{(*,b)}$ is a free $R$-module of rank $\binom{b+2}{2}$.

Both the symmetric algebra $\Sym(I)$ and Rees algebra $\rees(I)$ can be expressed  as  quotients of the  ring $S$. Defining a bihomogeneous surjection
 $S\to\Sym(I)$ by mapping $T_i \mapsto f_it$ induces isomorphisms
$$\Sym(I) \simeq S/L_1, \mbox{ where } L_1=\{\sum_{i=1}^n b_iT_i \ | \ \sum_{i=1}^n b_if_i=0\} \mbox{and} $$
$$ \rees(I) \simeq S/L, \mbox{ where } L=\{F(T_1,T_2,\ldots,T_n) \ | \ F(f_1,f_2,\ldots,f_n)=0\}.$$ 
 If the map $\Sym(I)\stackrel{}{\longrightarrow} \rees(I)$ gives an isomorphism between $\Sym(I)$ and $\rees(I)$, then $I$ is said to have  {\em linear type}.  Equivalently, $I$ has linear type if and only if $L=L_1$, that is the defining ideal of the Rees algebra $L$ is generated in bidegree $(*,1)$. This is the class of ideals where the Rees algebra is easiest to understand. 

The following result, essentially proved in \cite{refNS}, shows that $I$ has linear type in the case of interest for this paper, that is in the case of 3-generated ideals defining points in $\pr{2}$:
\begin{lem}\label{lintype}
Let $I$ be a strict almost complete intersection ideal defining a reduced set of points in $\pr{N}$. Then $I$ is an ideal of linear type.
\end{lem}
\begin{proof}
It is proven in \cite[Lemma 3.1]{refNS} that for  strict almost complete intersection ideals $I$ such that $R/I$ is equidimensional and satisfies $\depth R/I\geq \dim R/I-1$ (both of these conditions are clearly satisfied in case $I$ defines a reduced set of points, hence $R/I$ is arithmetically Cohen-Macaulay), the property that $I$ is of linear type is equivalent to the fact that $I$ is locally a complete intersection, that is $I_P$ is a complete intersection for every minimal prime $P\in\text{Ass}(R/I)$. For the class of ideals defining a finite reduced set of points i.e. $I=\cap_{i=1}^n I(p_i)$ as in the introduction,  we have for every point $p_i$ in the set that $I_P=I(p_i)$ is minimally generated by a regular sequence of $N$ linear forms.
\end{proof}

We use this nice property  to give an explicit description for the free resolutions of the square and the cube of a uniformly 3-generated ideals of points in $\pr{2}$.
 
 \begin{lem}\label{CI}
 Let $I$ be a strict almost complete intersection ideal defining a reduced set of points in in $\pr{2}$ and let $A=\left[\begin{matrix}P_1 & P_2 & P_3 \\ Q_1 & Q_2 & Q_3 \end{matrix}\right]^T$ be a presentation matrix for the module of syzygies on $I$  i.e. the Hilbert-Burch matrix of $I$. Then the symmetric and the Rees algebras of $I$ are given as quotients of the polynomial ring $S=R[T_1,T_2,T_3]$ by
 $$\rees(I)\simeq \Sym(I) \simeq S/(P_1T_1+P_2T_2+ P_3T_3, Q_1T_1+Q_2T_2+Q_3T_3).$$ 
 Furthermore, the defining ideal of these algebras, $(P_1T_1+P_2T_2+ P_3T_3, Q_1T_1+Q_2T_2+Q_3T_3)$ is a complete intersection.
 \end{lem}
 \begin{proof}
By the Hilbert-Burch Theorem  \cite[Theorem 20.15]{refEis},  $I$ is generated by the 2 by 2 minors of some $2\times 3$ matrix whose entries are homogeneous polynomials of the same degree in each of the two columns. Denote this matrix by $A=\left[\begin{matrix}P_1 & P_2 & P_3 \\ Q_1 & Q_2 & Q_3 \end{matrix}\right]^T$.
By definition, the symmetric algebra of $I$ is  given by the quotient in the statement and the isomorphism between the symmetric and the Rees algebra follows from Lemma \ref{lintype}. It remains to show that the ideal $(P_1T_1+P_2T_2+ P_3T_3, Q_1T_1+Q_2T_2+Q_3T_3)$ is a complete intersection. Since the two syzygies of $I$ are algebraically independent, the height of this ideal is two. This yields the desired conclusion that the defining ideal is a complete intersection.
\end{proof}
 
\begin{prop}\label{res}
Let $I$ be strict almost complete intersection ideal with minimal generators of the same degree $d$ defining a reduced set of points in $\pr{2}$. Let $A=\left[\begin{matrix}P_1 & P_2 & P_3 \\ Q_1 & Q_2 & Q_3 \end{matrix}\right]^T$ be a presentation matrix for the module of syzygies on $I$ (i.e. the Hilbert-Burch matrix of $I$). Let $d_0$ and $d_1$ denote the respective degrees of the polynomials in each of the two columns of $A$. Then the minimal free resolution of $I^2$ and $I^3$ are as follows: 

$$0 \lra R(-3d) \stackrel{X}{\lra} R(-2d-d_0)^3 \oplus R(-2d-d_1)^3 \lra R(-2d)^6 \lra I^2 \lra 0$$

$$0 \lra R(-4d)^3 \stackrel{Y}{\lra}  R(-3d-d_0)^6 \oplus R(-3d-d_1)^6\lra R(-3d)^{10} \lra I^3 \lra 0,$$
and the last homomorphisms in the respective resolutions can be described by the matrices $X$ and $Y$ given below: 
$$X= \left[\begin{array}{cccccccccccc} P_1 & P_2 & P_3 & -Q_1, -Q_2, -Q_3    \end{array}\right]^T,$$
$$ Y=\left[\begin{array}{cccccccccccc}P_1 & P_2 & P_3 & 0 & 0 & 0 & -Q_1 & -Q_2 & -Q_3 & 0 & 0 & 0 \\
                                             0 & P_1 & 0 & P_2 & P_3 & 0 & 0 & -Q_1 & 0 & -Q_2 & -Q_3 & 0 \\
                                            0 & 0 & P_1 & 0 & P_2 & P_3 & 0 & 0 & -Q_1 & 0 & -Q_2 & -Q_3 \\
                                               \end{array}\right]^T.$$
\end{prop} 
\begin{proof}
Resolving the complete intersection $\rees(I)$ over $S$ we obtain:
\begin{equation*}\label{reesres}
0 \lra S(-2d,-2) \stackrel{M}{\lra} S(-d-d_1,-1) \oplus S(-d-d_0,-1) \lra S \lra \rees(I) \lra 0.
\end{equation*}
Restricting this complex to its strands of degree $(*,2)$ and $(*,3)$  and using that $S_{(*,0)}\simeq R$, $S_{(*,1)}=R\ T_1\oplus R\ T_2 \oplus R\ T_3 \simeq R^3$ and $S_{(*,1)}=\oplus_{1\leq i \leq j\leq3}R\ T_iT_j \simeq R^6$, one obtains the resolutions stated in this Proposition. Furthermore, Lemma \ref{CI} yields that the map of $S$-modules labeled by $M$ in the resolution of the complete intersection $ \rees(I) $ can be chosen to be 
$$M=\left[\begin{array}{cc}P_1T_1+P_2T_2+ P_3T_3 & -(Q_1T_1+Q_2T_2+Q_3T_3) \end{array}\right]^T.$$
Writing the $R$-module homomorphisms induced by $M$ on the aforementioned strands of the complex in terms of the canonical $R$-bases $\mathcal{B}_0=\{1\}$ of $S_{(*,0)}$, $\mathcal{B}_1=\{T_1, T_2, T_3\}$  of $S_{(*,1)}$ and $\mathcal{B}_2=\{T_1^2, T_1T_2, T_2^2, T_2T_3, T_3^2\}$ for $S_{(*,2)}$ yields the descriptions of $X$ and $Y$ given in the statement of this Proposition.
\end{proof}

\begin{rem}
The idea of using Rees algebra techniques to provide an explicit description of minimal free resolutions of all ordinary powers of a uniformly 3-generated ideal of points in $\pr{2}$ will be exploited further in \cite{refNaSe}. In particular, a formula for the minimal free resolutions and for the Castelnuovo-Mumford regularity of all ordinary powers of  a uniformly 3-generated ideal of points in $\pr{2}$ will appear in full detail in \cite{refNaSe}. For the purposes of this note, we only require knowledge of the resolutions of $I^2$ and $I^3$ and a good command of the maps appearing therein, as  illustrated in Proposition \ref{res}.
\end{rem}

\section{A homological criterion for  the containment $I^{(m)}\subseteq I^r$}\label{sect3}

In this section, let $m=(x,y,z)$ be the graded maximal ideal of $R=\field[x,y,z]$ and let $m\geq r>0$ be integers. We view the containment $I^{(m)}\subseteq I^r$ through the lens of the natural map between the local cohomology modules $H^0_m(R/I^r)\to H^0_m(R/I^r)$. Interpreting the dual of this map from a homological perspective will yield a criterion for establishing (the failure of) the containment $I^{(m)}\subseteq I^r$.

\begin{prop}\label{crit}
Let $I$ be a homogeneous ideal. Consider the short exact sequence
$$0 \longrightarrow I^r/I^m \longrightarrow R/I^m \stackrel{\pi}{\longrightarrow} R/I^r \longrightarrow 0.$$The following statements are equivalent:
\begin{enumerate}
\item $I^{(m)}\subseteq I^r$
\item the induced map $H^0_m(\pi):  H^0_m(R/I^m) \to  H^0_m(R/I^r)$ is the zero homomorphism
\item  the induced map $\Ext^3(\pi):  \Ext^3_R(R/I^r,R) \to  \Ext^3_R(R/I^m,R)$ is the zero homomorphism
\end{enumerate}
\end{prop}
\begin{proof}
$(1)\Leftrightarrow (2):$ The map in $(2)$ can be rewritten explicitly, using the isomorphisms $H^0_m(I^r)=\frac{I^{(r)}}{I^r}$ and $H^0_m(I^m)=\frac{I^{(m)}}{I^m}$,  as $$H^0_m(\pi): \frac{I^{(m)}}{I^m} \to \frac{I^{( r )}}{I^r}.$$ Thus the map $H^0_m(\pi)$ is can be viewed as the composition of the inclusion $I^{(m)}\hookrightarrow I^{( r  )}$ with the canonical projection $ I^{( r )} \to \frac{I^{( r )}}{I^r}$. It follows that the image of this map  is  the zero module if and only if $\frac{I^{( m )}}{I^r}=0$, that is if and only if $I^{(m)}\subseteq I^r$.

$(2)\Leftrightarrow (3):$ The short exact sequence in the statement yields a long exact sequence which features the maps in (2) and (3) of this Lemma  as labeled below 
  $$\ldots \to\Ext^2_R(R/I^m,R)\to\Ext^2_R(I^r/I^m,R)\stackrel{\delta}{\to}\Ext^3_R(R/I^r,R) \stackrel{\Ext^3(\pi)}{\longrightarrow}  \Ext^3_R(R/I^m,R) \to  \Ext^3_R(I^r/I^m,R) \to 0.$$
  The equivalence $(2)\Leftrightarrow (3)$ follows from local duality, since for every integer $t$  there exist natural vector space isomorphisms between $H^0_m(R/I^r)_t\simeq \Ext^3_R(R/I^r,R(-3))_{-t-3}$ and $H^0_m(R/I^m)_t\simeq \Ext^3_R(R/I^m,R(-3))_{-t-3}$ under which the map $\Ext^3(\pi)_{-t-3}$ becomes the dual of the map $H^0_m(\pi)_t$. Thus either one of these maps is zero if and only if the other one is.
  \end{proof}

Next we give an effective way of checking the validity of the equivalent statements in Proposition \ref{crit} for the case when $I$ is a strict almost complete intersection defining a reduced set of points in $\pr{2}$, $I$ is generated in a single degree $d$ and $m=3, r=2$. 
    
    
\begin{prop}\label{inj}
Let $I$ be a homogeneous ideal with three minimal generators of the same degree $d$, defining a reduced set of points in $\pr{2}$ over a field of characteristic not equal to 3. Then the map $\Ext^3_R(R/I^2,R) \to  \Ext^3_R(R/I^3,R)$ is zero if and only if  the vector whose entries are the minimal generators of $I$ is in the image of the dual of the last non-zero map in the resolution of $R/I^3$. 
\end{prop}

\begin{proof}
By Proposition \ref{res}, the ideals $I^2$ and $I^3$ admit minimal free resolutions described below 

\begin{equation*}
\begin{CD} 
0 @>>>R(-4d)^3 @>Y>>  R(-3d-d_0)^6 \oplus R(-3d-d_1)^6@>>>R(-3d)^{10} @>>> I^3 @>>> 0\\
 @.	@VV\mu_1V	@VV\mu_2V @VV\mu_3V	@VV3\iota V	@.\\ 
0 @>>>R(-3d) @>>>  R(-2d-d_0)^3 \oplus R(-2d-d_1)^3@>>>R(-2d)^6 @>>> I^2 @>>> 0\\
  \end{CD}
  \end{equation*}
  
Let $f,g,h$ be a set of minimal generators of $I$ such that a bijection between $S/L_1$ and $\Sym(I)$ is given by $T_1\mapsto f, T_2\mapsto g, T_3\mapsto h$.  We denote by $\mu_i:S_{(*,i)}\to S_{(*,i-1)}$ the downgrading homomorphism, namely the $R$-module map which is given on an $R$-basis of $S_{(*,i)}$  by $\mu_i(T_{k_1}\ldots T_{k_i})=\sum_{j=1}^i f_jT_{k_1}\ldots \widehat{T_{k_j}} \ldots T_{k_i}$.   If $\iota: I^3 \hookrightarrow I^2$ denotes the natural inclusion, it is easy to see that $\mu_3$ induces the map $3\iota:I^3\to I^2$ via the projection $S\to \Sym(I)$, making the rightmost square of the above diagram commute. Furthermore, using the notation of Lemma \ref{CI} and setting $F=P_1T_1+P_2T_2+P_3T_3$ and $G=Q_1T_1+Q_2T_2+Q_3T_3$  to be the minimal generators of the defining ideal of $\rees(I)$, one has $\mu_1(F)=\mu_1(G)=0$ because $P_1f+P_2g+P_3h=Q_1f+Q_2g+Q_3h=0$. Consequently, extending the fact that for monomials $m\in S_{(*,j)}$ and $m'\in S_{(*,k)}$ with $i=j+k$ one has $\mu_i(mm')=\mu_j(m)m'+m\mu_k(m')$, we obtain
\begin{eqnarray*}
\mu_i(Fs) &=& \mu_1(F)s+F\mu_{i-1}(s)=F\mu_{i-1}(s)\mbox{ and }\\
\mu_i(Gs) &=& \mu_1(G)s+G\mu_{i-1}(s)=G\mu_{i-1}(s),\mbox{ for all } s\in S_{(*,i-1)}.
\end{eqnarray*}
Since the maps in the resolutions of powers of $I$ are induced by componentwise multiplication by $F$ or $G$, according to Theorem \ref{res},  one obtains a commutative diagram of $R$-modules and $R$-module homomorphisms as illustrated in the beginning of the proof. 

Since the map $\iota$ lifts the homomorphism $\pi$ of Proposition \ref{crit}, we have by homological shifting that  $\Ext^3(\pi)=\Ext^2(\iota)$. The matrix representing $\mu_1^*$ with respect to the canonical $R$-bases of $S_{(*,1)}$ and $S_{(*,0)}$  is $[f \ g \ h ]$. Therefore, after dualizing, the map $\Ext^2(3\iota)$ is given by $1\mapsto [f \ g \ h ]^T$. If $3$ is invertible then  $\Ext^2(\iota)=3^{-1} \Ext^2(3\iota)$, thus the map $\Ext^3(\pi)=\Ext^2(\iota)$ is injective if and only if  the vector $[f \ g \ h ]^T$ is not zero in $\Ext^3(R/I^3,R)$. Since $\Ext^3(R/I^3,R)=R(-4d)^3/ \rm{Image}({Y^T})$, the map $\Ext^3(\pi)$ is the zero map  if and only if the vector $[f \ g \ h ]^T$  is  contained in the image of $Y^T$, where $Y$ is the last non-zero map in the minimal free resolution of $I^3$ and $Y^T=\Hom(Y,R)$. \end{proof}

Combining the results of this section we obtain an effective version of our criterion:

\begin{thm}\label{main}
Let $I$ be a 3-generated homogeneous ideal with minimal generators $f,g,h$ of the same degree $d$, defining a reduced set of points in $\pr{2}$ over a field of characteristic not equal to 3. Set  $Y$ to be the matrix representing the last homomorphism in the the minimal free resolution of $I^3$, as described in Proposition \ref{res}:
$$0\longrightarrow R^3\stackrel{{Y}}{\longrightarrow}R^{12}\longrightarrow  R^{10} \longrightarrow I^3 \longrightarrow 0.$$
Then $I^{(3)} \subseteq (I^2)$ if and only if  $\left[f \ g \ h \right]^T \in \rm{Image}({Y^T})$.
\end{thm}
\begin{proof}
The Theorem follows from  Propositions \ref{crit} and \ref{inj}.
\end{proof}

\begin{rem}
Theorem \ref{main} can be regarded also as a computationally efficient criterion. Indeed, if successful, this criterion involves checking non-containment of a fixed element, namely $\left[f \ g \ h \right]^T$ in  a module generated in degrees at most $d-1$, namely $\rm{Image}({Y^T})$. By contrast, for saturating $I^3$ one has to deal with generators of  degree $3d$ and  usually  a well determined element is not readily available in $I^{(3)}$ for which to test containment. However, in the examples of section \ref{sect4} one  does have such a candidate. In fact, it is true that the product of the lines appearing in the respective configurations is always an element of $I^{(3)}$ but not of $I^2$. See \cite{refDST} and \cite{refHS} for a proof of this for the Fermat family. The author has verified computationally that the same is true for the Klein and Wiman configurations.
\end{rem}

With an eye towards the applications in section \ref{sect4}, we end with establishing a setup that allows to verify the hypothesis of Theorem \ref{main}. The resulting criterion is stated quite technically below, but the underlying idea is relatively simple: we know the explicit form of the matrix $Y$ that plays a prominent role in our criterion in Theorem \ref{main} in terms of the entries of the Hilbert-Burch matrix of $I$. We can also express the generators $f,g,h$ of $I$ in terms of the entries of the Hilbert-Burch matrix. Then checking whether $[f \ g \ h]^T\in \Im (Y^T)$ amounts to checking if there exist homogeneous polynomials $w_1,\ldots w_{12}$ satisfying the following matrix equality:
$$\left[\begin{matrix}P_2Q_3-P_3Q_2 \\ P_3Q_1-P_1Q_3\\P_1Q_2-P_2Q_1\end{matrix}\right]=\left[\begin{array}{cccccccccccc}P_1 & P_2 & P_3 & 0 & 0 & 0 & -Q_1 & -Q_2 & -Q_3 & 0 & 0 & 0 \\
                                             0 & P_1 & 0 & P_2 & P_3 & 0 & 0 & -Q_1 & 0 & -Q_2 & -Q_3 & 0 \\
                                            0 & 0 & P_1 & 0 & P_2 & P_3 & 0 & 0 & -Q_1 & 0 & -Q_2 & -Q_3 \\
                                               \end{array}\right] [w_1\ldots w_{12}]^T. $$
If each of the sets $\{P_1, P_2, P_3\}$, $\{Q_1, Q_2, Q_3\}$ and $\{P_2Q_3, P_3Q_2, P_3Q_1, P_1Q_3, P_1Q_2, P_2Q_1\}$ consists of linearly independent forms (this property is most clearly evident in the case when the entries of the Hilbert-Burch matrix are monomials, as in section \ref{sFermat}), the problem can be reduced to standard linear algebra and a contradiction can be obtained.  This is the main idea of Proposition \ref{6=0} below.

\begin{rem}
It is easy to see that in characteristic 2 and 3 the matrix equation above always has solutions. In characteristic 2 a solution is given by
$$ [w_1,\ldots, w_{12}]=[0,Q_3, Q_2 , 0 ,Q_1, 0, 0, 0, 0, 0, 0, 0]$$
and  in characteristic 3, a solution is given by $$ [w_1,\ldots, w_{12}]=[0,Q_3, 0 , 0 ,-Q_1, 0, 0, P_3, 0, 0,-P_1, 0].$$
However, the conclusions that can be derived from this fact with the help of Theorem \ref{main} are different. If $\char(\field)=2$, our theorem allows to conclude that the containment $I^{(3)}\subseteq I^2$ always holds for three-generated ideals $I$ defining points in $\PP_{\field}^2$. This recovers a special case of \cite[Remark 8.4.4]{refPSC}. By contrast, in characteristic 3, Theorem \ref{main} is inconclusive. 
\end{rem}

In the following, $R_i$ denotes the set of homogeneous polynomials of degree $i$ in $R$ and we use similar notation for homogeneous graded components of ideals. 

 \begin{prop}\label{6=0}
With the notation and hypotheses of Theorem \ref{main}, assuming the characteristic of the polynomial ring is not 2 or 3 and setting $\{[P_1,P_2,P_3]^T, [Q_1,Q_2,Q_3]^T\}$ to be  a set of minimal generators of degrees $d_0$ and $d_1$ respectively of the syzygy module of $I$, if  the conditions below are satisfied, then the vector $[f \ g \ h]^T$ is not in the image of the $R$-module map defined by the transpose of the matrix $Y$:
\begin{enumerate}
\item the forms $P_iQ_j$, $1\leq i,j\leq 3, i\neq j$ form a  linearly independent set
\item  there is a basis $\mathcal{B}$  for $R_d$ ($d=d_0+d_1$)
 containing the set of forms in (1) and such that
$$c(\phi_1,P_2Q_3)-c(\phi_1,P_3Q_2)+c(\phi_2,P_3Q_1)-c(\phi_2,P_1Q_3)+c(\phi_3,P_1Q_2)-c(\phi_3,P_2Q_1)=0,$$
 \end{enumerate}
 where $c(\phi,P_iQ_j)$ denotes the coefficient of $P_iQ_j$ in the polynomial $\phi\in R_d$ written in base $\mathcal B$, and $\phi_1, \phi_2, \phi_3$ are defined by
$$
\left[\begin{matrix}\phi_1 \\ \phi_2\\ \phi_3\end{matrix}\right]=
\left[ \begin{matrix} w_1 & w_2 & w_3 \\ w_2 & w_4 &w_5 \\ w_3 & w_5 & w_6 \end{matrix}\right] \left[ \begin{matrix}  P_1 \\ P_2 \\ P_3 \end{matrix}\right] +
\left[ \begin{matrix} w_7 & w_8 & w_9 \\ w_8 & w_{10} &w_{11} \\ w_9 & w_{11} & w_{12} \end{matrix}\right] \left[ \begin{matrix}  Q_1 \\ Q_2 \\ Q_3 \end{matrix}\right] $$
for arbitrary  homogeneous polynomials $w_1,\ldots, w_6$ of degree $d_1$ and $w_7,\ldots, w_{12}$ of degree $d_0$.

\end{prop}
\begin{proof}
Suppose  that $[f \ g \ h]^T=Y^T[w_1 \ldots w_{12}]^T$, where $w_1,\ldots w_6\in R_{d_1}$ and $w_7,\ldots,w_{12}\in R_{d_0}$. Recall from Proposition \ref{res} that the matrix $Y$ at the end of the resolution for $I^3$ has the following description:
$$ Y=\left[\begin{array}{cccccccccccc}P_1 & P_2 & P_3 & 0 & 0 & 0 & -Q_1 & -Q_2 & -Q_3 & 0 & 0 & 0 \\
                                             0 & P_1 & 0 & P_2 & P_3 & 0 & 0 & -Q_1 & 0 & -Q_2 & -Q_3 & 0 \\
                                            0 & 0 & P_1 & 0 & P_2 & P_3 & 0 & 0 & -Q_1 & 0 & -Q_2 & -Q_3 \\
                                               \end{array}\right]^T.$$
It is easy to see now that the equality $[f \ g \ h]^T=Y^T[w_1 \ldots w_{12}]^T$ can be rewritten as
$$\left[\begin{matrix}P_2Q_3-P_3Q_2 \\ P_3Q_1-P_1Q_3\\P_1Q_2-P_2Q_1\end{matrix}\right]=
\left[ \begin{matrix} w_1 & w_2 & w_3 \\ w_2 & w_4 &w_5 \\ w_3 & w_5 & w_6 \end{matrix}\right] \left[ \begin{matrix}  P_1 \\ P_2 \\ P_3 \end{matrix}\right] +
\left[ \begin{matrix} w_7 & w_8 & w_9 \\ w_8 & w_{10} &w_{11} \\ w_9 & w_{11} & w_{12} \end{matrix}\right] \left[ \begin{matrix}  Q_1 \\ Q_2 \\ Q_3 \end{matrix}\right]. $$
Taking the coefficients of the forms $P_iQ_j$, $1\leq i,j\leq 3, i\neq j$ in the equations above we have:
\begin{eqnarray*}
1 &=& c(\phi_1,P_2Q_3) \\
-1&=& c(\phi_1,P_3Q_2) \\
1 &=& c(\phi_2,P_3Q_1) \\
-1 &=& c(\phi_2,P_1Q_3) \\
1 &=& c(\phi_3,P_1Q_2)\\
-1 &=&c(\phi_3,P_2Q_1).
\end{eqnarray*}
The alternating sum of the six equalities above yields $6=0$, a contradiction in any characteristic different from 2 or 3.
\end{proof}

\begin{rem}
All results of sections \ref{sect2} and \ref{sect3}  remain true with the same proofs if one replaces throughout the requirement that $I$ be a 3-generated homogeneous ideal with generators of the same degree $d$, defining a reduced set of points in $\pr{2}$ with the weaker hypothesis that $I$ is an unmixed height two, locally complete intersection 3-generated homogeneous ideal of $\field[x,y,z]$, with generators of the same degree $d$. 
\end{rem}

\section{Application to the case of Fermat and Klein configurations}\label{sect4}

\subsection{The Fermat family}\label{sFermat}
Let $n\geq 3$ be and integer and let $\field$ be a field with $\char(\field)\neq 2$ containing  $n$ distinct $n^{\rm{th}}$ roots of 1. Consider the family of ideals $$I=(x(y^n-z^n), y(z^n-x^n),z(x^n-y^n))\subseteq \field[x,y,z].$$  These ideals were introduced in \cite{refDST} and further discussed in \cite{refBCH}, \cite{refHS}. In \cite[Proposition 2.1]{refHS} it is proved that these ideals define a  reduced set of $n^2+3$ points in $\pr{2}$. In \cite{refDHNSST} the configurations of points described by this family of ideals  are termed  Fermat configurations.

\begin{prop}\label{Fermat}
Let $I$ be the defining ideal of a Fermat configuration over a field $\field$ of characteristic $\char(\field)\neq 2, 3$. Then $I^{(3)} \not \subseteq I^2$.
\end{prop}
\begin{proof}
By definition, the ideals of the Fermat configuration have three minimal generators, all of degree $d=j+1$. Furthermore, as remarked in \cite{refDHNSST}, the minimal free resolution of $F$ has Hilbert-Burch matrix given by $\left[\begin{matrix} x^{n-1} &  y^{n-1} &  z^{n-1} \\  yz & xz & xy \end{matrix}\right]^T$. It is clear for $P_1=x^{n-1}, P_2= y^{n-1}, P_3=z^{n-1})$ and $Q_1=yz, Q_2=xz, Q_3=xy$  that condition (1) of Proposition \ref{6=0} is satisfied since the pairwise products are distinct monomials. 

Set $\mathcal B$ to be the monomial basis of $R_{n+1}$ and let the function $c(\psi,m)$ represent the coefficient of the monomial $m$ in a homogeneous polynomial $\psi$ of the same degree as $m$. With the notation in the statement of Proposition \ref{6=0} we have, using the function $c(-)$ on the right hand side of the the identities below to mean the coefficient of the monomial $P_i$ or $Q_j$ respectively in a homogeneous polynomial written in terms of the standard monomial basis of $R_2$ or $R_{n-1}$ respectively: 
\begin{eqnarray*}
 c(\phi_1,P_2Q_3)&=& c(w_2,Q_3)+c(w_9,P_2) \\
c(\phi_1,P_3Q_2)&=& c(w_3,Q_2)+c(w_8,P_3) \\
c(\phi_2,P_3Q_1)&=& c(w_5,Q_1)+c(w_8,P_3) \\
c(\phi_2,P_1Q_3) &=& c(w_2,Q_3)+c(w_{11},P_1) \\
c(\phi_3,P_1Q_2) &=& c(w_3,Q_2)+c(w_{11},P_1)\\
c(\phi_3,P_2Q_1) &=& c(w_5,Q_1)+c(w_9,P_2).
\end{eqnarray*}
Taking the alternating sum of these identities verifies condition (2) of Proposition \ref{6=0}. Combining Proposition \ref{6=0} with Theorem \ref{main} yields the desired result $I^{(3)} \not \subseteq I^2$.
\end{proof}

The result in Proposition \ref{Fermat} has previously  been proved first in \cite{refDST} for $n=3$ and later in \cite{refHS} for all $j$, based on the method of \cite{refDST}.  As new results, we are able to prove theoretically that another configuration, presented below, provides a counterexample to  $I^{(3)} \subseteq I^2$.

\subsection{The Klein configuration}\label{sKlein}

The Klein configuration of points (and lines) has been first introduced by Klein \cite{refKl}, Burnside \cite{refBu}, Coxeter \cite{refCo} and extensively studied in a number of later works, for example \cite{refGr}. The recent preprint  \cite{Bea} is the first to point out and verify computationally that for the Klein configuration the symbolic cube is not contained in the ordinary square of the defining ideal. In this section, we apply our criterion to give a conceptual proof of this fact.

  Although the Klein configuration is usually defined over the complex field (indeed, it is a noteworthy property that it cannot be realized in the real plane), it also makes sense to define it over finite fields, as noted in \cite{refGr} where the configuration over the field $\mathbb{F}_7$ is first introduced. 

Here we generalize the classical setting by defining the  Klein configuration over an arbitrary field (of any characteristic) as follows: let $\field$ be a field and consider the quadratic equation $t^2+t+2=0$.  Let $c$ and $\bar c$ be the roots of this equation in the splitting field $\FF$ of this equation over $\field$ (i.e. if $c,\bar c \in K$, then $\FF=K$, otherwise  $\FF=K[t]/(t^2+t+2)$). The Klein configuration  in $\PP_\FF^2$ consists of 49 $\FF$-rational points given by taking arbitrary permutations of the coordinates of the following points
\begin{equation}\label{klein}
\begin{matrix}
(1:0:0), & (1:\pm 1:0), & (1:\pm1: \pm \bar c) \\
(c:\pm 1:0), & (1:\pm 1:\pm 1), & (\bar c ^2: \pm 1: \pm 1). 
\end{matrix}
\end{equation}

These points arise as the reduced singular locus of an arrangement of 21 lines which are Hermitian dual to the 21 points listed in the first row of (\ref{klein}), namely the coefficients of these 21 lines are obtained by by taking arbitrary permutations of the following triples:
$$
\begin{matrix}
\left[1:0:0\right], & [1:\pm 1:0], & [1:\pm1: \pm c] .
\end{matrix}
$$
A remarkable property of the Klein configuration is that among the 49 points, the 21 listed in the first row of (\ref{klein}) are quadruple points for this line arrangement, while the 28 points listed in the second row of (\ref{klein}) are triple points. It is in fact the absence of double points that  clearly indicates why this configuration cannot be realized in the real projective plane. According to a theorem of Sylvester-Gallai, given a finite number of points in the real  plane, not all collinear, there is a line which contains exactly two of the points. The projective dual, version of this theorem then insures that any arrangement of lines in the real projective plane has at least one double point (a point that belongs to only two  of the lines), provided the lines are not all collinear at a single point.

We now show that the defining ideal of the Klein configuration fits into the class of strict almost complete intersections in most cases.

\begin{prop}\label{gensKlein}
Assume $\char(\FF)\neq 2,7$. Then the defining ideal of the Klein configuration in $\PP_\FF^2$ is minimally generated by three polynomials of degree 8:
\begin{eqnarray*}
f &=& xy(x^2-y^2)(4x^4+4y^4+(3c+9)x^2y^2+(5c-1)x^2z^2+(5c-1)y^2z^2+(15c+25)z^4),\\
g &=& yz(y^2-z^2)((15c+25)x^4+4y^4+(5c-1)x^2y^2+(5c-1)x^2z^2+(3c+9)y^2z^2+4z^4),\\
h &=& zx(z^2-x^2)(4x^4+(15c+25)y^4+(5c-1)x^2y^2+(3c+9)x^2z^2+(5c-1)y^2z^2+4z^4).
\end{eqnarray*}
\end{prop}
\begin{proof}
Consider among the points listed in (\ref{klein}) those that do not lie on $V(xy(x+y)(x-y))$. They have the following coordinates
$$
\begin{matrix}
 (1: \pm \bar c: \pm1), &  (\pm \bar c: \pm1: 1) \\
(c:\pm 1:0), & (\pm 1: c: 0) \\
(\pm 1: \bar c ^2:  \pm 1), & (\bar c ^2: \pm 1: \pm 1). \\
\end{matrix}
$$
We'll start by finding a form that vanishes on the squares of the coordinates of these points. Note that after taking squares of the coordinates the points can be grouped into 6 classes:  $$(1: \bar c^2: 1),  (\bar c^2: 1: 1), (c^2: 1:0), (1: c^2: 0), (1: \bar c ^4:  1), (\bar c ^4:  1:  1).$$ 

The remarkable thing about these six points is that they lie on a quadratic curve, namely 
$V(4x^2+4y^2+(3c+9)xy+(5c-1)xz+(5c-1)yz+(15c+25)z^2)$.
A conceptual reason why this is the case is the following: since this set of points is closed under switching $x$ and $y$ coordinates, it suffices to find a non-zero linear combination of $x^2+y^2, xy,(x+y)z$ and $z^2$ that vanishes on the points $(1: \bar c^2: 1), (c^2: 1:0), (1: \bar c ^4:  1)$; this will automatically also vanish on the other three points. Since the points impose 3 independent conditions on the 4 forms $x^2+y^2, xy,(x+y)z$ and $z^2$, it is clear that such a linear combination can be found. Having an explicit description of the coefficients as above helps identify the finite list of characteristics to be excluded from our argument. Specifically, we shall see below that our arguments do not apply for $\char(\mathbb{F})=2$ or $7$.

 It follows from the previous arguments that the form $xy(x+y)(x-y)C_3$ vanishes on all the points of (\ref{klein}), where
  $$C_3=4x^4+4y^4+(3c+9)x^2y^2+(5c-1)x^2z^2+(5c-1)y^2z^2+(15c+25)z^4$$ is obtained by substituting every variable by its square in the quadratic form above.
By symmetry, the forms $yz(y+z)(y-z)C_1$ and $zx(z+x)(z-x)C_2$ also vanish on these points, where $C_1$ is obtained from $C_3$ by substituting $x$ for $z$ and $z$ for $x$ and $C_2$ is obtained from $C_3$ by substituting $y$ for $z$ and $z$ for $y$.

Let $I=(xy(x+y)(x-y)C_3, yz(y+z)(y-z)C_1, zx(z+x)(z-x)C_2)\subset R=\FF[x,y,z]$. We wish to show this is the defining ideal of the Klein configuration. First note that this ideal is contained by construction in the defining ideal of the Klein configuration. It remains to show that $I$ is unmixed of the correct multiplicity. 
As a consequence of the manner in which $C_1$ has been defined, the polynomial $C_3-C_1$ is divisible by $x^2-z^2$ and similarly  $C_3-C_2$ is divisible by $y^2-z^2$ and  $C_1-C_2$ is divisible by $x^2-y^2$. Set
\begin{equation}\label{Ds}
\begin{split}
C_1-C_2 & =  (x^2-y^2)D_3 \\
C_2-C_3 & =  (y^2-z^2)D_1 \\
C_3-C_1 & =  (z^2-x^2)D_2 .
\end{split}
\end{equation}
It is easily checked  using the equations above that $$(zD_3)xy(x+y)(x-y)C_3+(xD_1)yz(y+z)(y-z)C_1+(yD_2) zx(z+x)(z-x)C_2)=0.$$  
thus $zD_3,xD_1,yD_2$ is a syzygy on the generators of $I$. 

Explicitly $D_3=(15c+21)x^2+(15c+21)y^2+(2c-10)z^2$ and the formulas for $D_1,D_2$ are obtained by symmetry. 
A determinant computation shows that $D_1, D_2,D_3$ are linearly independent if  and only if $c$ satisfies one of the equations $c+1=0$ or $13c+31=0$ in addition to $c^2+c+2=0$. The second linear equation is the only one compatible with the quadratic equation and they  can only have a common solution if $\char(\FF)= 2$ or $7$. It is easy to note that, except for these two characteristics mentioned above, the common vanishing locus of the polynomials $zD_3,xD_1,yD_2$ is empty if $2c-10\neq0$ (the equality $2c-10=0$ can only occur together with $c^2+c+2$ in characteristic 2),  thus the ideal $zD_3,xD_1,yD_2$ is primary to the maximal ideal and hence the three polynomials form a regular sequence. Therefore, by  \cite{refEH} or \cite{refTo}, the presence of a syzygy given by a regular sequence implies that  the quotient $R/I$ is Cohen-Macaulay, hence unmixed of Krull dimension one, since the three generators share no common factors. 

To finish the proof one needs to show that the multiplicity of $R/I$ is 49. First we show that the syzygy $zD_3,xD_1,yD_2$ is minimal. Suppose the contrary, then the triple $zD_3,xD_1,yD_2$ would be a multiple of a linear or quadratic minimal syzygy, thus contradicting that $zD_3,xD_1,yD_2$ form a regular sequence. Therefore  we have a minimal free resolution of $R/I$ of the form  
$$0  \to R(-11) \oplus R(-13) \to R(-8)^3 \to R \to R/I \to 0.$$
One easily computes the multiplicity of $R/I$ to be 49 using this exact sequence.
\end{proof}

\begin{rem}
For $\FF=\ZZ/7$ and $c=3$, the generators listed in the previous Proposition become  $xy(x^6-y^6), xz(x^6-z^6), yz(y^6-z^6)$. These three degree 8 equations are satisfied by all 57 points of $\PP_\FF^2$, meaning that the ideal of the Klein configuration in characteristic 7 has additional minimal generators. This is because the syzygy described in the proof above fails to yield a regular sequence in this case. Indeed, with the notation in the proof, one has $D_1=D_2=D_3=x^2+y^2+z^2$, thus the sequence $zD_3,xD_1,yD_2$ is not regular; it is in fact  a multiple of the minimal linear syzygy $z,x,y$.\end{rem}

\begin{thm}\label{Klein}
If $I$ is the defining ideal of a Klein configuration over a field $\FF$ of characteristic not equal to  $2, 3$ or $7$, then $I^{(3)} \not \subseteq I^2$.
\end{thm}
\begin{proof}
The proof of Proposition \ref{gensKlein} establishes that $I$ has is an almost complete intersection ideal uniformly generated in degree 8, with a minimal syzygy of degree $3$ 
\begin{equation} \label{Ps}
P_1=zD_3, P_2=xD_1, P_3=yD_2
\end{equation}
 given by a regular sequence. Thus the other minimal generator $[Q_1, Q_2,Q_3]^T$ for the module of syzygies on $I$ must be of degree 5. Let $P=(P_1,P_2,P_3)$ and $Q=(Q_1, Q_2, Q_3)$.
We now explicitly exhibit a minimal generator of degree 5 for the module of syzygies on $I$:
\begin{eqnarray*}
zC_2xy(x+y)(x-y)C_3+xC_2yz(y+z)(y-z)C_1+y((y^2-z^2)D_2+C_3)zx(z+x)(z-x)C_2 &=& \\
xyz((x^2-y^2)C_2C_3+(y^2-z^2)C_1C_2+(y^2-z^2)(C_3-C_1)C_2+(z^2-x^2)C_2C_3) &=&\\
xyz((x^2-y^2)C_2C_3+(y^2-z^2)C_1C_2+(y^2-z^2)C_2C_3-(y^2-z^2)C_1C_2+(z^2-x^2)C_2C_3) &=& 0.
\end{eqnarray*}
Thus we may set
\begin{equation}\label{Qs}
\begin{split}
Q_1 &= zC_2 \\
Q_2 &= xC_2 \\
Q_3 &= y((y^2-z^2)D_2+C_3)
\end{split}
\end{equation}

We summarize some important properties satisfied by the forms $P_1,P_2,P_3,Q_1,Q_2, Q_3$ that will be used in the proof. 
\begin{enumerate}
\item[{{$($a$)$}}] $Q_1,Q_2, Q_3$ span a complementary space to the homogeneous component of $P$ in degree 5
i.e. $Q_5\oplus P_5=R_5.$
\item[{{$($b$)$}}] the polynomials $\{P_iQ_j | 1\leq i \leq 3, 1\leq j \leq 3\}$ form a linearly independent set
\item[{{$($c$)$}}] the set of polynomials in (b) together with a basis for the homogeneous component of $P^2$ in degree 8 form a basis for $R_8$ i.e. $\Span\{P_iQ_j | 1\leq i \leq 3, 1\leq j \leq 3\}\oplus P^2_{\ 8}=R_8.$
\end{enumerate}

Property $($a$)$ can be verified by showing that no non-zero linear combination of the generators of $Q$ is in $P$. Note that all monomials in $Q_1, Q_2, Q_3$ have odd exponents for the variables $z,x,y$ respectively and even exponents for the other two variables. Since $P_1, P_2, P_3$ exhibit the same pattern, we have that  a linear combination $\sum_{j=1}^3\lambda_jQ_j\in P$ if and only if $\lambda_jQ_j\in P$ by separating the monomials in each of the three categories. It can be verified directly that that $Q_1,Q_2,Q_3 \not \in P$, so the linear combination must be trivial. Since $P$ is generated by a regular sequence it follows that  $\dim_{\mathbb{F}} P_5=18=\dim_{\mathbb{F}} R_5-3$, thus $Q_5$ and $P_5$ are complementary subspaces of $R_5$.
 
 To check that $($b$)$ and hence condition (1)  of Proposition \ref{6=0} holds, we note that  a linear dependence relation between the forms $P_iQ_j$ yields a syzygy on $P_1, P_2, P_3$.  Since the $P_i$ form a regular sequence,  this syzygy must lie in the module of Koszul syzygies on $P$. Since  no linear combination of the $Q_i$ lies in $P$, this means that the linear dependence relation is trivial. Property $($c$)$ follows from $($a$)$ coupled with a similar syzygy argument. 
 
 
 Next we take $\mathcal B_5$ to the the basis for $R_5$ expressed in {{$($a$)$}} and $\mathcal B_8$ to the the basis for $R_8$ expressed in $($c$)$ and we check that condition (2) of Proposition \ref{6=0} is satisfied with respect to $\mathcal B_8$. Towards this end, we need to compute the coefficients of polynomials of the form $\omega P_i$ and $\omega' Q_j$ when written in base $\mathcal B_8$, where $\omega\in R_5$ and $\omega'\in R_3$.

 Notice that $c(\phi P_k,P_iQ_j)=\delta_{ik}c(\phi,Q_j)$ for any $\phi\in R_5$, where $c(\phi P_k,P_iQ_j)$ is the coefficient with respect to the basis $\mathcal B_8$ and $c(\phi,Q_j)$ is the coefficient with respect to the basis $\mathcal B_5$. This is because by $($a$)$ any $\phi\in R_5$ can be written with respect to the base $\mathcal B_5$ as  $\phi=\sum_{j=1}^3 c(\phi,Q_j)Q_j+\phi'$ with $\phi'\in P$, whence the expression $\phi P_k=\sum_{j=1}^3 c(\phi,Q_j)P_kQ_j+\phi'P_k$, in which $\phi'P_k\in P^2$, gives the decomposition of $\phi P_k$ with respect to  the two complementary spaces of $R_8$ described in $($c$)$.
 Thus
\begin{equation}\label{eq7}
\begin{split}
 c(w_1P_1+w_2P_2+w_3P_3,P_2Q_3) = & c(w_2,Q_3) \\
c(w_1P_1+w_2P_2+w_3P_3,P_3Q_2 )= & c(w_3,Q_2) \\
c(w_2P_1+w_4P_2+w_5P_3,P_3Q_1) = & c(w_5,Q_1) \\
c(w_2P_1+w_4P_2+w_5P_3,P_1Q_3) = & c(w_2,Q_3) \\
c(w_3P_1+w_5P_2+w_6P_3,P_1Q_2) = & c(w_3,Q_2)\\
c(w_3P_1+w_5P_2+w_6P_3,P_2Q_1) = & c(w_5,Q_1)
\end{split}
\end{equation}
and the alternating sum of the right hand side terms is 0.

Computing the coefficients of $P_iQ_j$ in expressions of the form $\phi Q_j$ with $\phi \in R_3$ turns out to be somewhat more complicated. Using {\em Macaulay2} one obtains the following data, where the entry in each cell is the coefficient of the  basis element indexing the respective row in the expression of the polynomial corresponding to the column. We have scaled all of these coefficients by dividing them by a factor of $2^77$. The table gives the results after this scaling.
  \medskip
  
 \arraycolsep=2.3pt 
 
$\begin{array}{c|cccccccccc}
& x^3Q_1 & x^3Q_2 & x^3Q_3 & x^2yQ_1 & x^2yQ_2 & x^2yQ_3 & x^2zQ_1 & x^2zQ_2 & x^2zQ_3 & xy^2Q_1 \\ \hline
P_2Q_3&0&
      0&
      2 c+12&
      0&
      15 c-6&
      0&
      0&
      0&
      0&
      0\\
   P_3Q_2 &   0&
      0&
      15 c-6&
      0&
      15 c-6&
      0&
      0&
      0&
      0&
      0\\
   P_3Q_1 &  0&
      0&
      0&
      5 c-2&
      0&
      0&
      0&
      0&
      -8 c+16&
      0\\
  P_1Q_3  &  0&
      0&
      0&
      5 c-2&
      0&
      0&
      0&
      0&
      5 c-2&
      0\\
P_1Q_2  &    15 c-6&
      0&
      0&
      0&
      0&
      0&
      0&
      15 c-6&
      0&
      5 c-2\\
  P_2Q_1&    2 c+12&
      0&
      0&
      0&
      0&
      0&
      0&
      2 c+12&
      0&
      5 c-2\\
   \end{array}$

      \medskip
     
 $\begin{array}{c|cccccccccc}
& xy^2Q_2 & xy^2Q_3 & xyzQ_1 & xyzQ_2 & xyzQ_3 & xz^2Q_1 & xz^2Q_2 & xz^2Q_3 & y^3Q_1 & y^3Q_2 \\ \hline
P_2Q_3 &0&
15 c-6&
      5 c-2&
      0&
      0&
      0&
      0&
      5 c-2&
      0&
      -12 c-72\\
  P_3Q_2 &    0&
      15 c-6&
      5 c-2&
      0&
      0&
      0&
      0&
      -8 c+16&
      0&
      2 c+12\\
  P_3Q_1 &    0&
      0&
      0&
      5 c-2&
      0&
      0&
      0&
      0&
      2 c+12&
      0\\
  P_1Q_3  &   0&
      0&
      0&
      5 c-2&
      0&
      0&
      0&
      0&
      -12 c-72&
      0\\
  P_1Q_2 &    0&
      0&
      0&
      0&
      5 c-2&
      2 c+12&
      0&
      0&
      0&
      0\\
   P_2Q_1 &   0&
      0&
      0&
      0&
      5 c-2&
      15 c-6&
      0&
      0&
      0&
      0\\
      \end{array}$
      
        \medskip
      
$\begin{array}{c|cccccccccc}
& y^3Q_3 & y^2zQ_1 & y^2zQ_2 & y^2zQ_3 & yz^2Q_1 & yz^2Q_2 & yz^2Q_3 & z^3Q_1 & z^3Q_2 & z^3Q_3 \\ \hline
P_2Q_3 &   0&
0&
      0&
      0&
      0&
      5 c-2&
      0&
      0&
      0&
      0\\
   P_3Q_2 &   0&
      0&
      0&
      0&
      0&
      5 c-2&
      0&
      0&
      0&
      0\\
 P_3Q_1 &     0&
      0&
      0&
      15 c-6&
      15 c-6&
      0&
      0&
      0&
      0&
      15 c-6\\
   P_1Q_3 &   0&
      0&
      0&
      15 c-6&
      15 c-6&
      0&
      0&
      0&
      0&
      2 c+12\\
P_1Q_2&      0&
      0&
      5 c-2&
      0&
      0&
      0&
      0&
      0&
      2 c+12&
      0\\
  P_2Q_1 &    0&
      0&
      5 c-2&
      0&
      0&
      0&
      0&
      0&
      15 c-6&
      0\\
      \end{array}$
      
      \medskip
      
Taking alternating sums of pairs of consecutive rows yields the following identities for \\ $\phi_{1,Q}=w_7Q_1+w_8Q_2+w_9Q_3$, $\phi_{2,Q}=w_8Q_1+w_{10}Q_2+w_{11}Q_3$ and $\phi_{3,Q}=w_9Q_1+w_{11}Q_2+w_{12}Q_3$:
\begin{eqnarray*}
 c(\phi_{1,Q},P_2Q_3) -c(\phi_{1,Q} P_3Q_2) &=&\frac{1}{2^77}\left( -(13c-18)c(w_9,x^3)+(13c-18)c(w_9,xz^2) -(14c+84)c(w_9,y^3)\right)\\
c(\phi_{2,Q},P_3Q_1) - c(\phi_{2,Q}P1Q_3) &=&\frac{1}{2^77}\left( -(13c-18)c(w_{11},x^2z)+(14c+84)c(w_9,y^3)+(13c-18)c(w_{11},z^3)\right)\\
c(\phi_{3,Q},P_1Q_2) - c(\phi_{3,Q},P_2Q_1) &=&\frac{1}{2^77}\left( (13c-18)c(w_9,x^3)+(13c-18)c(w_{11},x^2z)-\right.\\
& &\left.(13c+18)c(w_9,xz^2)-(13c-18)c(w_{11},z^3) \right)
\end{eqnarray*}
where the function $c()$ on the right hand side of the equations above signifies coefficient with respect to the standard monomial basis of $R_3$. Taking the sum of the three identities above together with the equations in (\ref{eq7}) verifies condition (2) of Proposition \ref{6=0}, namely that 
$$c(\phi_1,P_2Q_3)-c(\phi_1,P_3Q_2)+c(\phi_2,P_3Q_1)-c(\phi_2,P_1Q_3)+c(\phi_3,P_1Q_2)-c(\phi_3,P_2Q_1)=0.$$
 An application of Proposition \ref{6=0} and Theorem \ref{main} now yields that $I^{(3)}\not\subseteq I^2$ for the ideal $I$ of the Klein configuration of points.
\end{proof}

\begin{rem}The criterion  of Theorem \ref{main} can be successfully applied to the Wiman configuration as well. The details of the argument in that case are however much more complicated than in the case of the Klein configuration, therefore we do not include them here.  
\end{rem}

\medskip

\section*{Acknowledgements} We thank Brian Harbourne, Tomasz Szemberg and Piotr Pokora for bringing the Klein and Wiman configurations to our attention and Uwe Nagel for  discussions regarding section \ref{sect2}. Part of the work contained  in this paper was done while the author was suported by an NSF-AWM grant. Computations with {\em Macaulay2} \cite{M2} were essential for our work. {\em Macaulay2} is freely available at \href{http://www.math.uiuc.edu/Macaulay2/}{{http://www.math.uiuc.edu/Macaulay2/}}.


\end{document}